%% file: cyc_mzsv.tex
\documentclass[10pt,a4paper,reqno]{amsart}

\input{ams_hack}

\makeatletter
\let\@@pmod\pmod
\DeclareRobustCommand{\pmod}{\@ifstar\@pmods\@@pmod}
\def\@pmods#1{\mkern4mu({\operator@font mod}\mkern 6mu#1)}
\makeatother

\usepackage[marginpar=2cm,ignoremp,margin=3cm]{geometry}

\usepackage{textcomp}

\usepackage[normalem]{ulem}
\usepackage{amssymb}

\usepackage{mathtools}
\usepackage{braket}
\usepackage{graphicx}

\newcommand{\leftsquigarrow}{\mathrel{\rotatebox[origin=c]{180}{$\rightsquigarrow$}}} 
\renewcommand{\epsilon}{\varepsilon}
\newcommand{\Q}{\mathbb{Q}}
\newcommand{\wt}{\mathrm{wt}}

\usepackage{thmtools}

\declaretheorem[
style=plain,
name=Theorem,
numberwithin=section,
refname={Theorem,Theorems},
Refname={Theorem,Theorems}
]{Thm}
\declaretheorem[
style=plain,
name=Proposition,
numberlike=Thm,
refname={Proposition,Propositions},
Refname={Proposition,Propositions}
]{Prop}
\declaretheorem[
style=plain,
name=Lemma,
numberlike=Thm,
refname={Lemma,Lemmas},
Refname={Lemma,Lemmas}
]{Lem}
\declaretheorem[
style=plain,
name=Claim,
numberlike=Thm,
refname={Claim,Claims},
Refname={Claim,Claims}
]{Clm}
\declaretheorem[
style=definition,
name=Definition,
numberlike=Thm,
refname={Definition,Definitions},
Refname={Definition,Definitions},
]{Def}
\declaretheorem[
style=definition,
name=Example,
numberlike=Thm,
refname={Example,Examples},
Refname={Example,Examples},
]{Eg}
\declaretheorem[
style=plain,
name=Conjecture,
numberlike=Thm,
refname={Conjecture,Conjectures},
Refname={Conjecture,Conjectures},
]{Conj}

\usepackage{hyperref}

\DeclareMathOperator{\Part}{Part}
\newcommand{\sym}{{\mathrm{sym}}}
\newcommand{\bl}{{\mathrm{bl}}}
\renewcommand{\vec}[1]{\mathbf{#1}}

\begin{document}
	
	\title[Analogues of cyclic insertion type identities for multiple zeta star values]{Analogues of cyclic insertion type identities \\ for multiple zeta star values}
	\date{June 27, 2018}
	\author{Steven Charlton}
	\address{Max-Planck-Institut f\"ur Mathematik \\
		Vivatsgasse 7 \\
		53111 Bonn \\
		Germany}
	
	\email{charlton@mpim-bonn.mpg.de}
	
	
	\begin{abstract}
		We prove an identity for multiple zeta star values, which generalises some identities due to Imatomi, Tanaka, Tasaka and Wakabayashi.  This identity gives an analogue of cyclic insertion type identities, for multiple zeta star values, and connects the block decomposition with Zhao's generalised 2-1 formula.
	\end{abstract}
	
	\maketitle
	
	\section{Introduction}
	
	Multiple zeta values (MZV's), and multiple zeta star values (MZSV's) are defined by the following series respectively
	\begin{align*}
		\zeta(s_1, \ldots, s_r) &\coloneqq \sum_{0 < n_1 < n_2 < \cdots < n_r} \frac{1}{n_1^{s_1} \cdots n_r^{s_r}} \, , \\
		\zeta^\star(s_1, \ldots, s_r) &\coloneqq \sum_{0 < n_1 \leq n_2 \leq \cdots \leq n_r} \frac{1}{n_1^{s_1} \cdots n_r^{s_r}} \, .
	\end{align*}
	In each case \( r \) is called the \emph{depth}, and \( s_1 + \cdots + s_r \) is called the \emph{weight}.  We make use of the shorthand `\( \wt \)' for the weight of the MZV's appearing in an identity, and write \( \{s\}^n \) to mean \( s \) repeated \( n \) times.  \medskip
	
	In \cite{imatomi2009some}, the following identities for multiple zeta star values are conjectured
	\begin{Conj}[Conjectures 4.1 and 4.3, \cite{imatomi2009some}]\label{conj:imatomi}
		For any integers \( a_0, \ldots, a_{2n} \geq 0 \) , we have
		\begin{gather*}
			\sum_{\text{permute \( a_0, \ldots, a_{2n} \)}} \zeta^\star(\{2\}^{a_0 + 1}, 1, \{2\}^{a_1}, 3, \{2\}^{a_2}, \ldots, 1, \{2\}^{a_{2n-1}}, 3, \{2\}^{a_{2n}}) \overset{?}{\in} \Q \pi^{\wt} \\
			\sum_{\text{permute \( a_1, \ldots, a_{2n} \)}} \zeta^\star(1, \{2\}^{a_1}, 3, \{2\}^{a_2}, \ldots, 1, \{2\}^{a_{2n-1}}, 3, \{2\}^{a_{2n}}) \overset{?}{\in} \Q \pi^{\wt} \, , 
		\end{gather*}
		 Notice the blocks of \( 2 \) all have lengths \( a_i \), except for the initial one; it has length \( a_0 + 1 \) in the first identity and length 0 in the second identity.
	\end{Conj}

	These identities are similar in structure to the cyclic insertion conjecture of \cite{borwein1998combinatorial}, on classical MZV's, and should perhaps be regarded as an analogue.  The cyclic insertion conjecture states
	\begin{Conj}[Conjecture 1, \cite{borwein1998combinatorial}]\label{conj:bbbl}
		For any integers \( a_0, \ldots, a_{2n} \geq 0 \) , we have
		\[
		\sum_{\text{cycle \( a_0, \ldots, a_{2n} \)}} \zeta(\{2\}^{a_0}, 1, \{2\}^{a_1}, 3, \{2\}^{a_2}, \ldots, 1, \{2\}^{a_{2n-1}}, 3, \{2\}^{a_{2n}}) \overset{?}{=} \frac{\pi^{\wt}}{(\wt + 1)!} \, .
		\]
	\end{Conj}

	In \cite{charlton2015zeta}, a symmetrised version of \autoref{conj:bbbl} was proven by the author, up to a rational, using the motivic MZV framework of Brown \cite{brown2012mixed,brown2012decomposition}.  This symmetrised result was generalised to a wider class of MZV's using the \emph{alternating block decomposition} of iterated integrals \cite{charlton2017alternating}, along with conjectural cyclic version, since proven exactly by Hirose and Sato. \medskip
	
	Zhao's generalised 2-1 formula, Theorem 1.4 in \cite{zhao2016identity}, gives an expression for an arbitrary MZSV as a sum of alternating MZV's, with arguments from a certain indexing set \( \Pi(\vec{s}^{(1)}) \).  As a consequence, Zhao gives a concise proof of \autoref{conj:imatomi}.  The goal of this paper is generalise Zhao's proof of \autoref{conj:bbbl}, by connecting \( \vec{s}^{(1)} \) with the `block decomposition' of the multiple zeta value \( \zeta(\vec{s}) \).  This allows us to give analogues of other MZV cyclic insertion identities in the MZSV case. \medskip
	
	Before stating the main result, we must first recall the construction of the block decomposition from \cite{charlton2017alternating}.  Any word in \( \{0,1\}^\times \) can be written as a concatenation of some number of `alternating words' \( 0, 1, 01, 10, 010, 101, 0101, 1010, \ldots \).  By deconcatenating \( w \) at a repeated letter, one obtains the (unique) decomposition of \( w \) into the minimal possible number of such words.  Moreover, by assuming \( w \) starts with 0, the lengths of the alternating words uniquely determine \( w \) since the concatenation occurs at a repeated letter.
	
	\begin{Def}[Block decomposition, \cite{charlton2017alternating}]
		For \( w \in \{0, 1\}^\times \), starting with a 0, write \( w \) as a concatenation of the fewest alternating words \( w_1, \ldots, w_n \), with lengths \( \ell_1, \ldots, \ell_n \) respectively.  The block decomposition of \( w \) is
		\[
			\bl(w) \coloneqq (\ell_1, \ldots, \ell_n) \, .
		\]
	\end{Def}

	\begin{Eg}
		For \( w = 01100101010010101 \), we have
		\[
			\bl(w) = \bl(\underbrace{01}_{2} \mid \underbrace{10}_{2} \mid \underbrace{0101010}_{7} \mid \underbrace{010101}_{6}) = (2, 2, 7, 6) \, .
		\]
		From the lengths \( (2, 2, 7, 6) \) we recover a unique word starting with 0, by writing
		\[
			\bl^{-1}(2, 2, 7, 6) = 0\underbracket{1 \mid 1}_{\mathclap{\text{repeat} \hspace{1em}}}\underbracket{0 \mid 0}_{\mathclap{\hspace{1em} \text{repeat}}}10101\underbracket{0 \mid 0}_{\text{repeat}}10101 = w
		\]
	\end{Eg}
	
	We can now state the main result of this paper.
		
	\begin{Thm}\label{thm:mzsvcyc}
		For integers \( \ell_i > 1 \), the following identity on MZSV's holds
		\[
		\sum_{\sigma \in S_n} \zeta^\star(\bl^{-1}( 2 \circ \ell_{\sigma(1)}, \ldots, \ell_{\sigma(n)}) ) = \sum_{\vec{r} \in \Part_{\text{odd}}(n)} 2^{\# r} \prod_i (\# r_i - 1)! \prod_i \widetilde{\zeta}\bigg(\sum\limits_{j \in r_i} \ell_j \bigg) \, .
		\]
		Here \( \Part_{\text{odd}}(n) \) is the set of partitions of \( \Set{1, \ldots, n} \) into odd-sized parts.  Moreover
		\[
		\zeta^\star(0 \, 10^{k_1-1} \cdots 10^{k_d-1} \, 1) \coloneqq \zeta^\star(k_1,\ldots,k_d) \, ,
		\]
		as in the iterated integral representation of an MZV, although this also includes the \emph{bounds} of integration.
		Finally, \( \widetilde{\zeta} \) and \( \circ \) are defined by
		\[
		\widetilde{\zeta}(n) = \begin{cases}
		\zeta(n) & \text{\( n \) odd} \\
		\frac{1}{2} \zeta^\star(\{2\}^{n/2}) & \text{\( n \) even,}
		\end{cases}
		\quad \text{ and } \quad
		\circ = \begin{cases}
		, & n \equiv \sum \ell_i \pmod*{2} \\
		+ & n \not\equiv \sum \ell_i \pmod*{2} \, .
		\end{cases}
		\]
		
		In particular, the sum is always a polynomial in Riemann zeta values, since \( \zeta^\star(\{2\}^n) \in \Q \pi^{2n} \).
	\end{Thm}

	In the case all \( \ell_i \) even, we recover the identities in \autoref{conj:imatomi}, and can give explicit terms in the right hand side in various cases.

	\begin{Eg}
		If \( (\ell_1, \ell_2, \ell_3) = (2a + 2, 2a + 2, 2a + 2) \), we are in the case \( n = 3 \), and \( \circ = \text{``\,+\,''} \).  Then
		\[
			\zeta^\star(\bl^{-1}(2 + 2a + 2, 2b + 2, 2c + 2)) = \zeta^\star(\{2\}^{a + 1}, 1, \{2\}^{b}, 3, \{2\}^{c}) \, .
		\]
		The theorem tell us that we need to sum over \( \vec{r} \in \Part_{\text{odd}}(3) = \Set{ 1\mid2\mid3 \, , 123 } \), and so we obtain
		\begin{align*}
		& {}{} = 2^3(1-1)!^3 \cdot \tfrac{1}{2} \zeta^\star(\{2\}^{a + 1}) \cdot \tfrac{1}{2} \zeta^\star(\{2\}^{b+1}) \cdot \tfrac{1}{2} \zeta^\star(\{2\}^{c+1}) \tag*{\( \leftsquigarrow \vec{r} = \{ 1 \mid 2 \mid 3 \} \)} + {}  \\[1ex]
		& \quad\quad {} + 2^1 (3-1)! \cdot \tfrac{1}{2} \zeta^\star(\{2\}^{a + b + c + 3}) \tag*{\( \leftsquigarrow \vec{r} = \{123\} \)}
		\end{align*}
		which simplifies to
		\[
		= \zeta^\star(\{2\}^{a+1})\zeta^\star(\{2\}^{b+1})\zeta^\star(\{2\}^{c+1}) + 2\zeta^\star(\{2\}^{a+b+c+3}) \, .
		\]
		
		This gives the identity
		\begin{align*}
			&\sum_{\text{permute \( a, b, c \)}} \zeta^\star(\{2\}^{a + 1}, 1, \{2\}^{b}, 3, \{2\}^{c}) \\
			& \quad = \zeta^\star(\{2\}^{a+1})\zeta^\star(\{2\}^{b+1})\zeta^\star(\{2\}^{c+1}) + 2\zeta^\star(\{2\}^{a+b+c+3}) \in \Q \pi^\wt \, ,
		\end{align*}
		as in case 1 of \autoref{conj:imatomi}.  \medskip
		
		If \( (\ell_1,\ldots,\ell_4) = (2a + 2, 2b + 2, 2c + 2, 2d+2) \), we are in the case \( n = 4 \), and \( \circ = \text{``\,,\,''} \).  Then
		\[
			\zeta^\star(\bl^{-1}(2, 2a + 2, 2b + 2, 2c + 2, 2d+2)) = \zeta^\star(1,\{2\}^{a}, 3, \{2\}^{b}, 1, \{2\}^{c}, 3, \{2\}^d) \, .
		\]
		We sum over \( \vec{r} \in \Part_{\text{odd}}(4) = \Set{1\mid234, 2 \mid 134, 3 \mid 124, 4 \mid 234, 1 \mid 2 \mid 3 \mid 4} \), and obtain
		\begin{align*}
			&\sum_{\text{permute \( a, b, c, d \)}} \zeta^\star(1,\{2\}^{a}, 3, \{2\}^{b}, 1, \{2\}^{c}, 3, \{2\}^d) \\
			&\quad = 2 \big( \zeta(\{2\}^{a + b + c + 3}) \zeta^\star(\{2\}^{d+1}) + \zeta(\{2\}^{a + b + d + 3}) \zeta^\star(\{2\}^{c+1}) + {} \\
			& \quad\quad {} + \zeta(\{2\}^{a + c + d + 3}) \zeta^\star(\{2\}^{b+1}) + \zeta(\{2\}^{b + c + d + 3}) \zeta^\star(\{2\}^{a+1}) \big) + {}  \\
			& \quad\quad {} + \zeta^\star(\{2\}^{a+1}) \zeta^\star(\{2\}^{b+1}) \zeta^\star(\{2\}^{c+1}) \zeta^\star(\{2\}^{d+1}) \in \Q \pi^\wt \, ,
			\end{align*}
			as in case 2 of \autoref{conj:imatomi}.  \medskip
	\end{Eg}

	\begin{Eg}[Hoffman's identity]
		For integers \( a, b, c \geq 0 \), Hoffman's identity on MZV's (generalised and proven up to \( \mathbb{Q} \) in \cite{charlton2017alternating}, and proven exactly in \cite{hirose2017hoffman}) states
		\begin{align*}
			&\zeta(\{2\}^a, 3, \{2\}^b, 3, \{2\}^c) - \zeta(\{2\}^b, 3, \{2\}^c, 1, 2, \{2\}^a) + \zeta(\{2\}^c, 1, 2, \{2\}^a, 1, 2, \{2\}^b) \\
			& \quad = -\zeta(\{2\}^{a+b+c+3}) \, .
		\end{align*}
		It arises from the block decomposition \( (\ell_1,\ell_2,\ell_3) = (2a + 3, 2b + 3, 2c + 2) \) of the first MZV above. \medskip
		
		We can apply \autoref{thm:mzsvcyc} to \( (\ell_1,\ell_2,\ell_3) = (2a + 3, 2b + 3, 2c + 2) \) to obtain an analogue on MZSV's.  We are in the case \( \circ = \text{``\,+\,''} \) , and we obtain the following combination of MZSV's.
	\begin{align*}
		& \zeta^\star(\{2\}^{a+1}, 3, \{2\}^b, 3, \{2\}^c) && {} + \zeta^\star(\{2\}^{b+1}, 3, \{2\}^a, 3, \{2\}^c) + {} \\
		&{} + \zeta^\star(\{2\}^{b + 1},3,\{2\}^c,1,2, \{2\}^{a}) && {} + \zeta^\star(\{2\}^{a + 1},3,\{2\}^c,1,2, \{2\}^{b}) + {} \\
		&{} + \zeta^\star(\{2\}^{c + 1}, 1,2, \{2\}^a, 1, 2, \{2\}^b) \hspace{-1em}  && {} + \zeta^\star(\{2\}^{c + 1}, 1,2, \{2\}^b, 1, 2, \{2\}^a)
	\end{align*}
		The theorem tell us that we need to sum over \( \vec{r} \in \Part_{\text{odd}}(3) = \Set{ 1\mid2\mid3 \, , 123 } \), and so we obtain
		\begin{align*}
			& {}{} = 2^3(1-1)!^3 \zeta(2a + 3)\zeta(2b + 3) \cdot \tfrac{1}{2}\zeta^\star(\{2\}^{c+1}) \tag*{\( \leftsquigarrow \vec{r} = \{ 1 \mid 2 \mid 3 \} \)} + {}  \\[1ex]
			& \quad\quad {} + 2^1 (3-1)! \cdot \tfrac{1}{2} \zeta^\star(\{2\}^{a + b + c + 4}) \tag*{\( \leftsquigarrow \vec{r} = \{123\} \)}
		\end{align*}
		which simplifies to
		\[
		= 4 \zeta(2a + 3)\zeta(2b + 3)\zeta^\star(\{2\}^{c+1}) + 2 \zeta^\star(\{2\}^{a + b + c + 4})
		\]
	\end{Eg}

	Similar identities can be given for a wide range of initial block lengths, allowing one to produce identities for many MZSV's with indices 1, 2, 3.  For example
	\begin{Eg}
		Starting with \( \zeta^\star(1, 3, 3, \{2\}^m) \), one reads off the block decomposition
		\[
		\bl^{-1}(0 ; 1 \, 100 \, 100 \, (10)^m ; 1) = (2, 2, 3, 2m+2) \, .
		\]
		By taking \( (\ell_1, \ell_2, \ell_3) = (2, 3, 2m+2) \), we are in the case \( \circ = \text{``\,,\,''} \) as \( \sum \ell_i \equiv 3 \pmod*{2} \).  The theorem gives us the following identity containing \( \zeta^\star(1,3,3,\{2\}^m) \):
		\begin{align*}
			&\zeta^\star(1, 3, 3, \{2\}^m) + \zeta^\star(1, 2, 1, \{2\}^m, 3) + \zeta^\star(1, \{2\}^m,3, 3) + \\
			&\zeta^\star(1, 2, 1, 3, \{2\}^m) + \zeta^\star(1, 3, \{2\}^m, 1, 2) + \zeta^\star(1, \{2\}^m, 3, 1, 2) \\
			& \quad = 2 \zeta(2) \zeta(3) \zeta^\star(\{2\}^{m+1}) + 4 \zeta(2m + 7) \, .
		\end{align*}
	\end{Eg}
	
	\paragraph{Acknowledgements}  This work was completed during the trimester program ``Periods in Number Theory, Algebraic Geometry and Physics'' at the Hausdorff Institute for Mathematics, concurrent with the author's stay at the Max Planck Institute for Mathematics.  It was motivated by observations made during the author's stay at the MZV research centre, Kyushu university.  I am grateful all three institutes for their hospitality and excellent working conditions.
	
	I am also grateful to Nobuo Sato for helpful discussions during the trimester, and for directing me to Zhao's generalised 2-1 identity which plays a key role in the proof.
	
	\section{Proof of the theorem}
	
	Much of the proof relies on Zhao's generalisation of the 2-1 formula, proven in \cite{zhao2016identity}.  The key step is to relate the block decomposition to Zhao's \( \vec{s}^{(1)} \), with the following lemma.
	
	\paragraph{Warning:} Since I use the opposite contention for MZV's, the version of \( \vec{s}^{(1)} \) used here is the reverse of the one obtained from Zhao's definition.  These changes are incorporated into the proofs.
	
	\begin{Lem}
		For \( \circ = \text{``\,,\,'' or ``\,+\,''} \), let \( \vec{L} = (2 \circ \ell_1, \ldots, \ell_n) \) be a block decomposition with all \( \ell_i > 1 \), and corresponding MZV arguments \( \vec{s} = (s_1, \ldots, s_{n'}) \).  Then Zhao's \( \vec{s}^{(1)} \) associated to \( \vec{s} \) is given by
		\[
			\vec{s}^{(1)} = (\widetilde{\ell_1}, \ldots, \widetilde{\ell_i}) \, 
		\]
		where
		\[
			\widetilde{\ell_i} = \begin{cases}
				\ell_i & \text{if \( \ell_i \) odd} \\
				\overline{\ell_i} & \text{if \( \ell_i \) even}
			\end{cases}
		\]
		
		\begin{proof}
			Since block notation, and arguments strings are in \( 1 \leftrightarrow 1 \) correspondence (and block decompositions with \( \ell_i > 1 \) correspond to \( s_i \in \Set{1,2,3} \), with no \( 1, 1 \)), we can build up the block decomposition, and \( s^{(i)} \) term-by-term, and see they match as claimed above.
			
			Suppose that both \( s^{(2)} = (\widetilde{t_1},\ldots,\widetilde{t_k}) \) and \( \mathbf{L} = (t_1,\ldots,t_k) \) have been constructed for \( \zeta(s_1,\ldots,s_n) \), and that they agree as claimed.  Then for \( \zeta(s_1,\ldots,s_n,s_{n+1}) \) we obtain by Zhao's inductive definition
			\[
				s^{(1)} = \begin{cases}
					(\widetilde{t_1},\ldots,\widetilde{t_k}) \cdot (1) = (\widetilde{t_1},\ldots,\widetilde{t_k},1) & s_{n+1} = 1 \\
					(\widetilde{t_1},\ldots,\widetilde{t_k}) \oplus (2) = (\widetilde{t_1},\ldots,\widetilde{t_{k-1}},\widetilde{t_k} \oplus 2) & s_{n+1} = 2 \\
					(\widetilde{t_1},\ldots,\widetilde{t_k}) \oplus (\overline{1},\overline{2}) = (\widetilde{t_1},\ldots,\widetilde{t_{k-1}}, \widetilde{t_k} \oplus \overline{1}, \overline{2}) & s_{n+1} = 3
				\end{cases}
			\]
			If \( t_k \) is even, then \( \widetilde{t_k} = \overline{t_k} \), so we have \( \widetilde{t_k} \oplus 2 = \overline{t_k} \oplus 2 = \overline{t_k + 2} = \widetilde{t_k + 2} \), and \( \widetilde{t_k} \oplus 1 = \overline{t_k} \oplus \overline{1} = t_k + 1 = \widetilde{t_k + 1} \).  If \( t_k \) is odd, then \( \widetilde{t_k} = t_k \), so we have \( \widetilde{t_k} \oplus 2 = t_k \oplus 2 = t_k + 2 = \widetilde{t_k + 2} \), and \( \widetilde{t_k} \oplus 1 = t_k \oplus \overline{1} = \overline{t_k + 1} = \widetilde{t_k + 1} \).
			
			So
			\[
			s^{(1)} = \begin{cases}
			 (\widetilde{t_1},\ldots,\widetilde{t_k},1) & s_{n+1} = 1 \\
			(\widetilde{t_1},\ldots,\widetilde{t_{k-1}},\widetilde{t_k+2} & s_{n+1} = 2 \\
			(\widetilde{t_1},\ldots,\widetilde{t_{k-1}}, \widetilde{t_k + 1}, \widetilde{2}) & s_{n+1} = 3
			\end{cases}
			\]
			
			Now consider the corresponding construction on \( \mathbf{L} \).  If \( s_n = 1 \), then the integral word is
			\[
				\cdots 0 \, 1 \, ; \, 1 = \cdots01 \mid 1 \quad \leftrightarrow \quad \mathbf{L} = (\ldots, \ell_{n-1}, 1) \, .
			\]
			We cannot have \( s_{n+1} = 1 \), so we have either
			\begin{align*}
				s_{n+1} = 2 \quad &\rightsquigarrow\quad  \cdots01 \, \underline{10} \, ; \, 1 = \cdots 01 \mid 101 \\
				& \Rightarrow \quad \mathbf{L} = (\ldots, \ell_{n-1}, 3) = (\ldots, \ell_{n-1}, \ell_n + 2) \, .
			\end{align*}
			Or we have
			\begin{align*}
				s_{n+1} = 3 \quad &\rightsquigarrow \quad \cdots 01 \, \underline{100} \, ; \, 1 = \cdots 01 \mid 10 \mid 01 \\
				&\Rightarrow \quad \mathbf{L} = (\ldots, \ell_{n-1}, 2, 2) = (\ldots, \ell_{n-1}, \ell_n + 1, 2) \, .
			\end{align*}
			
			If \( s_n \geq 2 \), then the integral word is
			\[
				\ldots 10 ; 1 = \ldots 101 \quad \leftrightarrow \quad \mathbf{L} = (\ldots, \ell_n)
			\]
			We may have \( s_{n+1} = 1, 2, 3 \), so we obtain
			\begin{align*}
				&&& s_{n+1} = 1 && \rightsquigarrow \quad \ldots 10 \, \underline{1} \, ; \, 1 = 101 \mid 1 &&\quad \Rightarrow \quad \mathbf{L} = (\ldots, \ell_n, 1) &\\
				&&& s_{n+1} = 2 && \rightsquigarrow\quad  \ldots 10 \, \underline{10} \, ; \, 1 = 10101 &&\quad \Rightarrow \quad \mathbf{L} = (\ldots, \ell_n + 2) & \\
				&&& s_{n+1} = 3 && \rightsquigarrow\quad  \ldots 10 \, \underline{100} \, ; \, 1 = 1010 \mid 01 &&\quad \Rightarrow \quad \mathbf{L} = (\ldots, \ell_n + 1, 2) \, . &
			\end{align*}
			
			We see that after inserting the new argument \( s_{n+1} \), we still have matching between \( s^{(1)} \) and \( \mathbf{L} \).
			
			To complete the proof, we must check the base case holds.  If we start with \( \zeta(1) \) then
			\[
				s^{(1)} = (1) \quad \text{ and } \quad \zeta(1) = I(0;1;1) \quad \Rightarrow \quad \mathbf{L} = (2,1) 
			\]
			
			If we start with \( \zeta(s_1>1) \), then
			\begin{align*}
				s^{(1)} = (\{1\}^{s_1 - 2}, \widetilde{2}) \quad &\text{and} \quad \zeta(s_1) = I(0;10^{s_1-1};1) = I(010 \mid (0 \mid)^{s_1-3} \ldots \mid 01) \\
				&\Rightarrow \quad \mathbf{L} = (2+1, \{1\}^{s_1-3},2) \, .
			\end{align*}
			In both cases, \( \mathbf{L} = (2 \circ \ell_1,\ldots,\ell_n) \) matches with \( s^{(1)} \) as claimed, and the lemma is proven.
		\end{proof}
	\end{Lem}
	
	\begin{Lem}[Zhao]\label{lem:zhao}
		For any arguments \( \vec{s} \), we have
		\[
			\zeta^\star(\vec{s}) = \epsilon(\vec{s}) \sum_{\vec{p} \in \Pi(\vec{s}^{(1)})} 2^{\#\vec{p}} \zeta(\vec{p}) \, ,
		\]		
		where \( \Pi(s_1,\ldots,s_\ell) \) is the set of all indices of the form \( (s_1 \circ \cdots \circ s_\ell) \), where \( \circ \) is either ``\,,\,'' or ``\,$\oplus$\,'', and \( \epsilon(\vec{s}) = 1 \) if \( s_1 = 1 \), \( \epsilon(\vec{s}) = -1 \) if \( s_1 \geq 2 \).
		
		\begin{proof}
			Apply Theorem 1.4 of Zhao \cite{zhao2016identity}, and pass to the limit \( n \to \infty \) using Lemma 4.5 of Zhao \cite{zhao2016identity}.
		\end{proof}
	\end{Lem}

	\begin{Prop}
		Let \( \vec{s} = (s_1, \ldots, s_{n}) \) be given, and assume \( S_{n} \) acts on the indices \( 1, \ldots, n \) in the standard way.  Then
		\[
			\sum_{\sigma \in S_{n}} \sum_{\vec{p} \in \Pi(\sigma \cdot \vec{s})} 2^{\# \vec{p}} \zeta(\vec{p}) = \sum_{\vec{q} \in \Part(n)} 2^{\# \vec{q}} \prod_i (\# q_i)! \zeta_\sym\bigg(\bigoplus\limits_{j \in q_1} s_j, \ldots, \bigoplus\limits_{j \in q_t} s_j \bigg) \, ,
		\]
0		where \( \zeta_\sym(a_1,\ldots, a_t) \coloneqq \sum_{\tau \in S_td} \zeta(a_{\tau(1)}, \ldots, a_{\tau(t)}) \) symmetrises the arguments of the MZV. \medskip
		
		That is, one can move the \( S_{n} \) action from \( \mathbf{L} \) to the arguments of the zeta, at the expense of some coefficients.
		
		\begin{proof}
			Firstly, observe that \( p \in \Pi(\sigma \cdot \vec{s}) \) means \( p = s_{\sigma(1)} \circ \cdots \circ s_{\sigma(n)} \), where each \( \circ \) is \( , \) or \( \oplus \).  But this is equivalent to \( p = \sigma \cdot (s_1 \circ \cdots \circ s_n) \), under the induced \( S_n \) action, and \( (s_1 \circ \cdots \circ s_n) \in \Pi(s_1,\ldots,s_n) \).  So we can write
			\[
				= \sum_{\sigma \in S_n} \sum_{p \in \Pi(\vec{s})} 2^{\# \vec{p}} \zeta(\sigma \cdot  \vec{p}) \, .
			\]
			Warning: \( \sigma \) acts on the elements \( s_i \) inside \( \vec{p} = (s_1 \circ \cdots \circ s_n) \), and not on the comma separated blocks.  So \( \zeta(\sigma \cdot \vec{p}) \) is not simply \( \zeta_\sym(\vec{p}) \).  We need to do further manipulation to obtain the desired form. \medskip
			
			An element \( \vec{p}\in \Pi(\vec{s}) \) is of the form \( s_1 \circ \cdots \circ s_n \) for some choices \( \circ = \mathop{,} \text{ or } \oplus \).  That is
			\begin{align*}
				\vec{p} &= ( s_1 \oplus \cdots \oplus s_{i_1} , s_{i_1+1} \oplus \cdots \oplus s_{i_2} , \ldots,
				s_{i_{t-1}+1} \oplus \cdots \oplus s_{\underbrace{i_t}_{=n}} ) \text{ , and }\\
				\sigma \cdot \vec{p} &= ( s_{\sigma(1)} \oplus \cdots \oplus s_{\sigma(i_1)} , s_{\sigma(i_1+1)} \oplus \cdots \oplus s_{\sigma(i_2)} , \ldots,
				s_{\sigma(i_{t-1}+1)} \oplus \cdots \oplus s_{\underbrace{\sigma(i_t)}_{i_t=n}} ) \, .
			\end{align*}
			
			We can therefore define a surjective map
			\begin{align*}
				\phi \colon (\Pi(\vec{p}) , S_n) & \to \Part^\ast(n) \\
				(\vec{p}, \sigma) & \mapsto [\Set{\sigma(1),\ldots,\sigma(i_1)}, \Set{\sigma(i_1 + 1), \ldots, \sigma(i_2)}, \ldots, \Set{\sigma(i_{t-1}+1),\ldots,\sigma(i_t)}]
			\end{align*}
			where \( i_1,\ldots,i_t \) are given by the expression for \( \sigma \cdot \vec{p} \) above.  Here \( \Part^\ast(n) \) is the set-partitions of \( \{ 1, \ldots, n \} \), where the order of the elements of the parts is not important, but the order of the parts themselves is.  That is \( [\Set{a,b}, \Set{c}] = [\Set{b,a},\Set{c}] \), but these are different from \( [\Set{c}, \Set{a,b}] \).
			
			If \( q = \phi(\vec{p}, \sigma) = [q_1,\ldots,q_t] \), then \( \zeta(\sigma \cdot \vec{p}) = \zeta(\bigoplus_{j \in q_1} s_j, \ldots, \bigoplus_{j \in q_t} s_j) \), and \( \# \vec{q} = \# \vec{p} \).  Notice that \( \# \phi^{-1}(\vec{q}) = \#q_1! \cdots \#q_t! \), since any permutation which respects the parts of \( \vec{p}\) maps to the same \( \vec{q} \).  So we can write that the desired sum is
			\[
				= \sum_{\vec{q} \in \Part^\ast(n)} 2^{\# q} \left( \prod_i \#q_i! \right) \zeta\bigg(\bigoplus_{j \in q_1} s_j, \ldots, \bigoplus_{j \in q_t} s_j\bigg) \, .
			\]
			
			Finally, we have a surjective map
			\begin{align*}
				\psi \colon \Part^\ast(n) &\to \Part(n) \\
					[q_1,\ldots,q_t] &\mapsto \Set{q_1,\ldots,q_t} \, ,
			\end{align*}
			with \( \psi^{-1}(\Set{q_1,\ldots,q_t})  = \Set{ [q_{\sigma(1)},\ldots,q_{\sigma(t)}] | \sigma \in S_t } \).
			
			So the sum can be written
			\[
				= \sum_{\vec{q} \in \Part(n)} 2^{\# q} \#q_1! \cdots \#q_t! \underbrace{\sum_{\sigma \in S_t} \zeta\bigg(\bigoplus_{j \in q_{\sigma(1)}} s_j, \ldots, \bigoplus_{j \in q_{\sigma(t)}} s_j\bigg)}_{\eqqcolon \zeta_\sym} \, .
			\]
			This completes the proof.
		\end{proof}
		\end{Prop}
	
		Using the symmetric sum formula \cite{hoffman1992multiple} (or rather Zhao's generalisation to alternating MZV's in Lemma 5.1 of \cite{zhao2016identity}), one can evaluate the RHS above.  We obtain the following
	
		\begin{Prop}\label{prop:oddpartsum}
			The following evaluation holds
			\[
				\sum_{\vec{q} \in \Part(n)} 2^{\# q} \bigg( \prod_{i=1}^{\#\vec{q}} \#q_i! \bigg) \; \zeta_\sym\bigg(\bigoplus_{j \in q_{1}} s_j, \ldots, \bigoplus_{j \in q_{t}} s_j\bigg) = \sum_{\vec{r} \in \Part_{\text{odd}}(n)} 2^{\# \vec{r}} \prod_{i=1}^{\#\vec{r}} (\# r_i - 1)! \prod_{i=1}^{\#\vec{r}} \zeta\bigg(\bigoplus_{j \in r_i} s_j\bigg)
			\]
		\end{Prop}
	
		\begin{proof}
			Firstly, we must apply the symmetric sum theorem to evaluate the left hand side.  It gives
			\[
				= \sum_{\vec{q}\in \Part(n)} 2^{\#\vec{q}} \bigg( \prod_{i=1}^{\#\vec{q}} \# q_i! \bigg) \sum_{\vec{t} \in \Part(\#\vec{q})} (-1)^{\#\vec{q} - \#\vec{t}} \bigg( \prod_{j = 1}^{\#\vec{t}} (\# t_j - 1)! \bigg) \prod_{j=1}^{\#\vec{t}} \zeta\bigg( \bigoplus_{\alpha \in t_j} \bigoplus_{\beta \in q_\alpha} s_\beta \bigg) \, .
			\]
			As the parts \( q_\alpha \) are disjoint, the \( \zeta \) argument
			\[
				\bigoplus_{\alpha \in t_j} \bigoplus_{\beta \in q_\alpha} s_\beta
			\]
			can be written as
			\[
				\bigoplus_{\alpha \in r_j} s_\alpha \, .
			\]
			for some partition \( \vec{j} \in \Part(n) \).  This partition is obtained from \( (\vec{q}, \vec{t}) \) by `flattening' in the following sense
			\begin{align*}
				 f(\vec{q},\vec{t}) = & \vec{r} = \{ r_1,\ldots,r_{\#\vec{t}} \} \text{, where} \\
				& r_i = \bigcup_{j \in t_i} q_j \, .
			\end{align*}
			For example
			\begin{align*}
				& (\vec{q} = \Set{ q_1 = \Set{1, 2, 4} , q_2 = \Set{3, 5}, q_3 = \Set{6,8}, q_4 = \Set{7} }, \vec{t} = \Set{ \Set{1,3,4}, \Set{2} }) \\
				& \mapsto \Set{ q_1 \cup q_3 \cup q_4 , q_2 } = \Set{ \Set{1,2,4,6,7,8}, \Set{3,5} } \, .
			\end{align*}
			
			So we may formally write the sum as
			\[
				\sum_{\vec{r} \in \Part(n)} \sum_{(\vec{q},\vec{t}) \in f^{-1}(\vec{r})} 2^{\#\vec{q}} (-1)^{\#\vec{q} - \#\vec{t}} \bigg( \prod_{i=1}^{\#\vec{q}} \# q_i! \bigg) \bigg( \prod_{j = 1}^{\#\vec{t}} (\# t_j - 1)! \bigg) \prod_{k=1}^{\#\vec{r}} \zeta\bigg( \bigoplus_{\alpha \in r_k} s_\alpha \bigg) \, .
			\]
			We thus need to evaluate the coefficient
			\[
				c_\vec{r} \coloneqq \sum_{(\vec{q},\vec{t}) \in f^{-1}(\vec{r})} 2^{\#\vec{q}} (-1)^{\#\vec{q} - \#\vec{t}} \bigg( \prod_{i=1}^{\#\vec{q}} \# q_i! \bigg) \bigg( \prod_{j = 1}^{\#\vec{t}} (\# t_j - 1)! \bigg) \, .
			\] We want to show two things: firstly that if the partition \( \vec{r}\) has any even size parts, then the coefficient is 0.  Secondly if the partition only has odd size parts, the coefficient is as indicated in the statement of the proposition.
			
			Firstly, we can describe \( f^{-1}(\vec{r}) \) more explicitly, as follows.  The elements \( \vec{q} \) which flatten to \( \vec{r} \) are obtained as \( \vec{q} = \coprod_{i=j}^{\#\vec{r}} \vec{T}_j \), where \( \vec{T}_j \) is any partition of \( r_j \).  This choice of partitions \( \vec{T}_1,\ldots,\vec{T}_{\#\vec{r}} \) determines \( \vec{t} \), since \( r_j = \bigcup \vec{T}_j \).  For example
			\begin{align*}
				& \vec{r} = \Set{ \Set{1, 3, 4, 5, 7} , \Set{ 2, 6, 8} } \\
				& \xrightarrow{\text{\( f^{-1} \) contains}} \{ \underbrace{\Set{ 1,4 \mid 3 \mid 5,7 }}_{\vec{T}_1} , \underbrace{\Set{ 2,6 \mid 8 }}_{\vec{T}_2} \} = \Set{\Set{1,4}, \Set{2,6}, \Set{3}, \Set{5,7}, \Set{8}}
			\end{align*}
			and \( \vec{t} = \Set{\Set{1,3,4}, \Set{2,5} } \).
			
			Under this construction we have
			\begin{gather*}
				\#\vec{t} = \#\vec{r} \, , \quad \quad \# t_i = \# \vec{T_i} \, , \quad \quad \#\vec{q} = \sum_{i=1}^{\#\vec{r}} \# \vec{T}_i \\
				q_j = T_{k,l} \text{ , some \( k, l \), so that} \\
				\prod_{j=1}^{\#\vec{q}} \#q_j! = \prod_{k=1}^{\#\vec{r}} \prod_{l=1}^{\#\vec{T}_k} \#T_{k,l}!
			\end{gather*}
			Thus
			\[
					c_\vec{r} = (-1)^{\#\vec{r}}\sum_{T_1 \in \Part(\#r_1)} \cdots \sum_{T_{\#\vec{r}} \in \Part(\#r_{\#\vec{r}})} (-2)^{\sum_{i} \# T_i} \cdot \prod_{k=1}^{\#\vec{r}} \prod_{l=1}^{\#\vec{T}_k} \#T_{k,l}! \cdot \prod_{j = 1}^{\#\vec{r}} (\# T_j - 1)! \, .
			\]
			This sum can now be factored into a product of the following form
			\[
				c_\vec{r}= (-1)^{\#\vec{r}} \prod_{i=1}^{\#\vec{r}} g(\#r_i) \, ,
			\]
			where
			\[
				g(i) \coloneqq \sum_{\vec{w} \in \Part(i)} (-2)^{\#\vec{w}} \prod_{l=1}^{\#\vec{w}} \#w_l! \cdot (\#\vec{w} - 1)! \, .
			\]
			
			I claim that \( g \) can be evaluated as follows
			\[
				g(i) = \begin{cases} -2 (i-1)! & \text{ \( i \) odd } \\
					0 & \text{ \( i \) even.}
					\end{cases}
			\]
			If this claim does hold, then
			\begin{align*}
				c_\vec{r} &= \begin{cases}
					0 & \text{some \( \#r_i \) even} \\
					(-1)^{\#\vec{r}} \prod_{i=1}^{\#\vec{r}} (-2) (\#r_i-1)! & \text{all \( \#r_i \) odd}
				\end{cases} \\
				&= \begin{cases}
					0 & \text{some \( \#r_i \) even} \\
					2^{\#\vec{r}} \prod_{i=1}^{\#\vec{r}} (\#r_i-1)! & \text{all \( \#r_i \) odd.}
				\end{cases}
			\end{align*}
			So the proposition will follow.
		\end{proof}
	
		For the proof to be complete, we need to show the following claim.
	
		\begin{Clm}
			Let
			\[
				g(n) \coloneqq \sum_{\vec{w} \in \Part(n)} (-2)^{\#\vec{w}} \prod_{l=1}^{\#\vec{w}} \#w_l! \cdot (\#\vec{w} - 1)! \, ,
			\]
			then
			\[
				g(n) = \begin{cases} -2 (n-1)! & \text{ \( n \) odd } \\
				0 & \text{ \( n \) even.}
				\end{cases}
			\]
			
			\begin{proof}
				We show the generalised identity
				\[
					\frac{1}{n!} g(n, x) = -\frac{1}{n} + \frac{1}{n} (1 + x)^n \, ,
				\]
				where
				\[
					g(n,x) \coloneqq \sum_{\vec{w} \in \Part(n)} x^{\#\vec{w}} \prod_{l=1}^{\#\vec{w}} \#w_l! \cdot (\#\vec{w} - 1)!
				\]	
				Hence for \( x = -2 \), we obtain
				\[
					g(n,-2) = -(n-1)! + (n-1)! (-1)^n = \begin{cases}
					-2(n-1)! & \text{\( n \) odd} \\
						0 & \text{\( n \) even,}
					\end{cases}
				\]
				as claimed. \medskip
				
				To show the generalised identity, we can first show that the derivatives agree.  Then integrating gives
				\[
					\frac{1}{n!} g(n,x) = c + \frac{1}{n} (1 + x)^n \, ,
				\]
				for some constant \( c \).  One sees that \( c = -\frac{1}{n} \) by setting \( x = 0 \); the left hand side is 0 and the right hand side is \( c + \frac{1}{n} \), which proves the claim. \medskip
				
				To see the derivatives agree, we must show
				\begin{align}
 					& \frac{1}{n!} g'(n,x) = (1 + x)^{n-1} \, \text{ equivalently, } \nonumber \\
					& g'(n,x) = n! (1 + x)^{n-1} \label{eqn:gder} \, .
				\end{align}
				Term-by-term differentiation of \( g(n,x) \) gives
				\[
					g'(n,x) = \sum_{\vec{w} \in \Part(n)} x^{\#\vec{w}-1} \prod_{l=1}^{\#\vec{w}} \#w_l! \cdot \#\vec{w}! \, .
				\]
				The coefficient of \( x^{i-1} \) on the right hand side of \autoref{eqn:gder} is
				\[
					n! \binom{n-1}{i-1} \, .
				\]
				The coefficient of \( x^{i-1} \) on the left hand side is
				\[
					\sum_{\substack{\vec{w} \in \Part(n) \\
						\#\vec{w} = i}} \prod_{l=1}^{\#\vec{w}} w_l! \cdot \underbrace{\#\vec{w}!}_{i!}
				\]
				These two expressions give two different ways to count the number of ordered partitions of \( [1,\ldots,n] \) into \( i \) non-empty ordered parts, and hence are equal.  Here an ordered partition with ordered parts means that \( [[1,2],[3]] \), \( [[2,1],[3]] \), \([[3],[1,2]] \) and \( [[3],[2,1]] \) are all counted as distinct.  For simplicity, refer to such a partition as an ordered/ordered partition. \medskip
				
				We can form an ordered/ordered partition of \( [1,\ldots,n] \) into \( i \) parts by first taking any permutation of \( [1,\ldots,n] \), then inserting \( i-1 \) bars into any choice of the \( n-1 \) gaps, breaking the \( i \) non-empty parts.
				\begin{align*}
					[1,\ldots,8] & {} \xrightarrow[\phantom{\text{insert bars}}]{\text{permute}} [4\underset{\uparrow}{,} 5, 2, 3\underset{\uparrow}{,} 7, 6\underset{\uparrow}{,} 1, 8] \\
					& {} \xrightarrow{\text{insert bars}} [[4],[5,2,3],[7,6],[1,8]] \, .
				\end{align*}
				
				There are \( n! \) permutations, and \( \binom{n-1}{i-1} \) ways of choosing \( i - 1 \) positions from the \( n - 1 \) gaps.  This gives the right hand side.
				
				Alternatively, we can form an ordered/ordered partition of \( [1,\ldots,n] \) by taking a set-partition of \( \Set{1,\ldots,n} \) into \( i \) parts, then reordering the \( i \) parts arbitrarily, as well as arbitrarily reordering the elements of each part.  Every such ordered/ordered partition of \( [1,\ldots,n] \) arises in this way, for some unique \( \vec{w} \), as forgetting about both orderings gives a surjection onto \( \Part(n) \).  
				\begin{align*}
					\vec{w} \in \Part(n) & \xrightarrow[\phantom{\text{permute parts}}]{} \Set{\Set{1,8},\Set{2,3,5},\Set{4},\Set{6,7}} \\
					& \xrightarrow[\text{and elements}]{\text{permute parts}} [[4],[5,2,3],[7,6],[1,8]]
				\end{align*}
				Let \( \vec{w} \in \Part(n) \) be a (set-)partition of \( \Set{1,\ldots,n} \) into \( i \) parts, with sizes of each part \( \#w_1, \ldots, \#w_i \) respectively.  Then there are \( i! \prod_{l=1}^{\#\vec{w}} \#w_l! \) such ordered/ordered partitions arising from \( \vec{w} \).  We must sum over all such \( \vec{w} \in \Part(n) \), giving
				\[
					\sum_{\substack{\vec{w} \in \Part(n) \\ \#\vec{w} = i}}  i! \prod_{l=1}^{\#\vec{w}} \#w_1! \, .
				\]
				This is the left hand side. \medskip
				
				The coefficients of both sides agree, hence we get the required equality of derivatives, and so claim follows.
			\end{proof}
		\end{Clm}

		Finally, we can uses these results to prove the theorem.
		
		\begin{proof}[Proof of theorem]
			Using \autoref{lem:zhao}, we have that
			\[
				\sum_{\sigma} \zeta^\star(\bl^{-1}(2 \circ \ell_{\sigma(1)}, \ldots, \ell_{\sigma(n)})) = \epsilon(\circ) \sum_{\sigma \in S_n} \sum_{\vec{p} \in \Pi(\sigma \cdot (\widetilde{\ell_1}, \ldots, \widetilde{\ell_n}))} 2^{\# \vec{p}} \zeta(\vec{p}) \, .
			\]
			Notice \( \epsilon(\circ) \) matches with \( \epsilon(s) \), for the following reason.  If \( \circ = \text{``\,,\,''} \), then \( \zeta^\star(\bl^{-1}(2, \ell_1, \ldots)) = I(01 \mid 1 \cdots) = \zeta(1,\ldots) \) and \( \epsilon(\vec{s}) = 1 \) since \( s_1 = 1 \).  Otherwise \( \circ = ``\,+\,'' \), then \( \zeta_\star(\bl^{-1}(2+\ell_1,\ldots)) = I(0101\cdots) = \zeta(2,\ldots) \), and \( \epsilon(\vec{s}) = -1 \) since \( s_1 = 2 \).
			
			Interchange the summations, write the result as a sum over odd-sized partitions using the \autoref{prop:oddpartsum}
			\[
				 = 	\epsilon(\circ) \sum_{\vec{r} \in \Part_{\text{odd}}(n)} 2^{\# \vec{r}} \prod_i (\# r_i - 1)! \prod_i \zeta\Big(\bigoplus_{j \in r_i} \widetilde{\ell_j}\Big) \, .
			\]
			
			Since the size of each partition is odd, we can explicitly evaluate \( \bigoplus_{j \in p_i} \widetilde{\ell_j} \), and the resulting \( \zeta \) as follows.
			
			\paragraph{Case \( \# \{ \ell_i | \text{ even } \} \equiv 0 \pmod*{2} \):}  Then
			\[
				\bigoplus_{j \in p_i} \widetilde{\ell_j} = \sum_{j \in p_i} \ell_j \, ,
			\]
			and this is sum is odd.  This is because the number of bars is additive (i.e. the sign is multiplicative), and there are an even number of bars in total.  So the \( \oplus \) sum agrees with the sum of the undecorated \( \ell_i \).  Moreover, the total is odd since we sum an odd number of odd numbers.
			
			Overall, this means
			\[
				\zeta\bigg(\bigoplus_{j \in p_i} \widetilde{\ell_j} \bigg) = \zeta\bigg( \sum_{j \in p_i} \ell_j \bigg) \, .
			\]
			
			\paragraph{Case \( \# \{ \ell_i | \text{ even } \} \equiv 1 \pmod*{2} \):}  Then
			\[
				\bigoplus_{j \in p_i} \widetilde{\ell_j} = \overline{\sum_{j \in p_i} \ell_j} \, ,
			\]
			 and this sum is even.  This is because there are an odd number of bars in total, so one remains after doing the \( \oplus \)-sum.  Consequently the \( \oplus \)-sum agrees with the bar of the undecorated sum.  Moreover, the total is even, since we add an even number of odd numbers.
			 
			 This means
			 \[
			 	\zeta\bigg( \bigoplus_{j \in \pi_i} \widetilde{\ell_j} \bigg) = \zeta\bigg( \overline{\sum_{j \in p_i}} \ell_j \bigg) \, .
			 \]  We can now use Zlobin's evaluation \cite{zlobin2005brief} \( \zeta^\star(\{2\}^n) = -2\zeta(\overline{2n}) \),  (which is also contained in Zhao's \(2\)-\(1\) theorem) to write
			 \[
			 	\zeta\bigg( \overline{\sum_{j \in p_i}} \ell_j \bigg) = -\frac{1}{2} \zeta^\star\bigg(\{2\}^{\text{\normalsize $ \frac{1}{2} \sum\nolimits_{j \in p_i} \ell_j $}} \bigg) \, .
			 \] 
			 This is almost our definition of \( \widetilde{\zeta} \).  I claim that the number of \( -1 \) signs between \( \epsilon(\circ) \) and all \( -\zeta^\star(\{2\}^n) \)'s is even. We may discard it to obtain an equivalent formula with our original definition of \( \widetilde{\zeta} \). \medskip
			 
			 Why is the total number of \( -1 \)'s even?
			 
			 \paragraph{Case \( \circ = \text{``\,,\,''} \):} Here \( \epsilon(\mathop{,}) = 1 \), since the MZV's begin \( \zeta^\star(\bl^{-1}(2,\geq2)) = \zeta^\star(01\mid10\cdots) = \zeta^\star(1,\geq2) \).  I claim that the number of `even-sum' parts is even, hence the total number of \( -1 \)'s is even as claimed.  In this case \( n \equiv \sum \ell_i \pmod*{2} \), and we can check \( n \) odd or even separately.
			 
			 Suppose \( n \) odd, then \( \sum \ell_i \) also odd.  By counting the number \( n \) of \( \ell_i \), we see \( \vec{p} \in  \Part_{\text{odd}}(\ell_i) \) has an odd number of parts.  If an odd number of parts have even sum, we would have \( \sum \ell_i \) even, a contradiction.
			 
			 Similarly if \( n \) even, then \( \sum \ell_i \) also even.  By counting the number \( n \) of \( \ell_i \), we see \( \vec{p}\in \Part_{\text{odd}}(\ell_i) \) has an even number of parts.  If an odd number of parts have even sum, we would again have \( \sum \ell_i \) odd, a contradiction. \medskip
			 
			 \paragraph{Case \( \circ = \text{``\,+\,''} \):} Here \( \epsilon(\mathop{+}) = -1 \), since the MZV's begin \( \zeta^\star(\bl^{-1}(2+\geq2, \ldots)) = \zeta^\star(0101 \cdots) = \zeta^\star(2, \ldots) \).  I claim that the number of `even-sum' parts is odd, hence the total number of \( -1 \)'s is even as claimed.  In this case \( n \not\equiv \sum \ell_i \pmod*{2} \), so just check \( n \) odd or even separately.
			 
			 Suppose \( n \) odd, then \( \sum \ell_i \) is even.  By counting the number \( n \) of \( \ell_i \), we see that \( \vec{p}\in \Part_{\text{odd}}(\ell_i) \) has an odd number of parts.  If an even number of parts have even sum, we would obtain \( \sum \ell_i \) odd.
			 
			 Finally \( n \) even, so \( \sum \ell_i \) is odd.  By counting the number \( n \) of \( \ell_i \), we see that \( \vec{p}\in \Part_{\text{odd}}(\ell_i) \) has an even number of parts.  If an even number of parts have even sum, we would obtain \( \sum \ell_i \) even. \medskip
			 
			 In all cases the overall number of \( -1 \)'s is even and we can drop the \( -1 \) from the definition of \( \widetilde{\zeta} \), to obtain the required result.  This completes the proof of the theorem.
		\end{proof}
	
	\bibliography{cyc_mzsv}
	\bibliographystyle{plainurl}
	
\end{document}

%% file: ams_hack.tex
\iftrue
\makeatletter
\def\@settitle{%
  \vspace*{-2em}
  \begin{flushleft}%
    \LARGE\bfseries
    \strut\@title\strut
  \end{flushleft}%
}
\def\@setauthors{%
  \begingroup
  \def\thanks{\protect\thanks@warning}%
  \trivlist
  \raggedright
  \large \@topsep27\p@\relax
  \advance\@topsep by -\baselineskip
  \item\relax
  \author@andify\authors
  \def\\{\protect\linebreak}%
  \authors
  \ifx\@empty\contribs
  \else
    ,\penalty-3 \space \@setcontribs
    \@closetoccontribs
  \fi
  \normalfont
  \endtrivlist
  \endgroup
}
\def\@setaddresses{\par
  \nobreak \begingroup
  \small\raggedright
  \def\author##1{\nobreak\addvspace\smallskipamount}%
  \def\\{\unskip, \ignorespaces}%
  \interlinepenalty\@M
  \def\address##1##2{\begingroup
    \par\addvspace\bigskipamount\noindent
    \@ifnotempty{##1}{(\ignorespaces##1\unskip) }%
    {\ignorespaces##2}\par\endgroup}%
  \def\curraddr##1##2{\begingroup
    \@ifnotempty{##2}{\nobreak\noindent\curraddrname
      \@ifnotempty{##1}{, \ignorespaces##1\unskip}\/:\space
      ##2\par}\endgroup}%
  \def\email##1##2{\begingroup
    \@ifnotempty{##2}{\nobreak\noindent E-mail address%
      \@ifnotempty{##1}{, \ignorespaces##1\unskip}\/:\space
      \ttfamily##2\par}\endgroup}%
  \def\urladdr##1##2{\begingroup
    \def~{\char`\~}%
    \@ifnotempty{##2}{\nobreak\noindent\urladdrname
      \@ifnotempty{##1}{, \ignorespaces##1\unskip}\/:\space
      \ttfamily##2\par}\endgroup}%
  \addresses
  \endgroup
  \global\let\addresses=\@empty
}
\def\@setabstracta{%
    \ifvoid\abstractbox
  \else
    \skip@17pt \advance\skip@-\lastskip
    \advance\skip@-\baselineskip \vskip\skip@
    \box\abstractbox
    \prevdepth\z@ 
    \vskip-28pt
  \fi
}
\renewenvironment{abstract}{%
  \ifx\maketitle\relax
    \ClassWarning{\@classname}{Abstract should precede
      \protect\maketitle\space in AMS document classes; reported}%
  \fi
  \global\setbox\abstractbox=\vtop \bgroup
    \normalfont\small
    \list{}{\labelwidth\z@
      \leftmargin0pc \rightmargin\leftmargin
      \listparindent\normalparindent \itemindent\z@
      \parsep\z@ \@plus\p@
      
    }%
    \item[\hskip\labelsep\bfseries\abstractname.]%
}{%
  \endlist\egroup
  \ifx\@setabstract\relax \@setabstracta \fi
}

\def\ps@headings{\ps@empty
  \def\@evenhead{%
    \setTrue{runhead}%
    \normalfont\scriptsize
    \rlap{\thepage}\hfill
    \def\thanks{\protect\thanks@warning}%
    \leftmark{}{}}%
  \def\@oddhead{%
    \setTrue{runhead}%
    \normalfont\scriptsize
    \def\thanks{\protect\thanks@warning}%
    \rightmark{}{}\hfill \llap{\thepage}}%
  \let\@mkboth\markboth
}\ps@headings

\def\section{\@startsection{section}{1}%
  \z@{-1.4\linespacing\@plus-.5\linespacing}{.8\linespacing}%
  {\normalfont\bfseries\Large}}
\def\subsection{\@startsection{subsection}{2}%
  \z@{-.8\linespacing\@plus-.3\linespacing}{.5\linespacing\@plus.2\linespacing}%
  {\normalfont\bfseries\large}}
\def\subsubsection{\@startsection{subsubsection}{3}%
  \z@{.7\linespacing\@plus.2\linespacing}{-1.5ex}%
  {\normalfont\bfseries}}
\def\paragraph{\@startsection{paragraph}{4}%
  \z@{.7\linespacing\@plus.2\linespacing}{-1.5ex}%
  {\normalfont\itshape}}
\def\@secnumfont{\bfseries}

\renewcommand\contentsnamefont{\bfseries}
\def\@starttoc#1#2{\begingroup
  \setTrue{#1}%
  \par\removelastskip\vskip\z@skip
  \@startsection{}\@M\z@{\linespacing\@plus\linespacing}%
    {.5\linespacing}{
      \contentsnamefont}{#2}%
  \ifx\contentsname#2%
  \else \addcontentsline{toc}{section}{#2}\fi
  \makeatletter
  \@input{\jobname.#1}%
  \if@filesw
    \@xp\newwrite\csname tf@#1\endcsname
    \immediate\@xp\openout\csname tf@#1\endcsname \jobname.#1\relax
  \fi
  \global\@nobreakfalse \endgroup
  \addvspace{32\p@\@plus14\p@}%
  \let\tableofcontents\relax
}
\def\contentsname{Contents}
\def\l@section{\@tocline{2}{.5ex}{0mm}{5pc}{}}
\def\l@subsection{\@tocline{2}{0pt}{2em}{5pc}{}}
\makeatother
\fi 